\documentclass[12pt]{amsart}

\usepackage[utf8]{inputenc}

\usepackage{amsfonts}
\usepackage{amsmath}
\usepackage{amssymb}
\usepackage{mathtools}
\usepackage{tabularx}

\usepackage{geometry} 
\usepackage{pxfonts}
\usepackage{euscript}
\usepackage{bbold,bbm}
\usepackage{comment}
\usepackage{scalerel}
\usepackage{xcolor}

\usepackage{hyperref}
\usepackage{graphics}
\usepackage{epstopdf} 
\usepackage[all,2cell,arrow,matrix,tips]{xy} \UseAllTwocells \SilentMatrices
\usepackage{graphicx}
\usepackage{verbatim}
\usepackage{leftidx}
\usepackage{enumerate}
\usepackage{phonetic}
\usepackage{lipsum}
\usepackage{subfigure}
\usepackage{float}
\usepackage{extarrows}

\theoremstyle{definition}
\newtheorem{thm}{Theorem}[section]
\newtheorem{prop}[thm]{Proposition}
\newtheorem{cor}[thm]{Corollary}
\newtheorem{lem}[thm]{Lemma}
\newtheorem{rem}[thm]{Remark}
\newtheorem{example}[thm]{Example}

\newtheorem*{question}{Question}

\newcommand\Nb {\mathbb{N}}
\newcommand\Zb {\mathbb{Z}}

\newcommand\Cb {\mathbb{C}}

\newcommand\CA {\EuScript{A}}
\newcommand\CB {\EuScript{B}}
\newcommand\CC {\EuScript{C}}
\newcommand\CD {\EuScript{D}}

\newcommand\CH {\EuScript{H}}
\newcommand\CI {\EuScript{I}}
\newcommand\CJ {\EuScript{J}}
\newcommand\CK {\EuScript{K}}
\newcommand\CL {\EuScript{L}}

\newcommand\CS {\EuScript{S}}

\newcommand{\FI} {\mathfrak{I}}

\newcommand{\Clb} {\mathcal{B}}

\newcommand{\Clk} {\mathcal{K}}
\newcommand{\Cll} {\mathcal{L}}
\newcommand{\Clm} {\mathcal{M}}

\newcommand{\Clz} {\mathcal{Z}}

\DeclareMathOperator{\Alg}{Alg}
\DeclareMathOperator{\Lat}{Lat}
\DeclareMathOperator{\id}{id}
\DeclareMathOperator{\Ad}{Ad}
\DeclareMathOperator{\wot}{wot}
\DeclareMathOperator{\Wot}{WOT}
\DeclareMathOperator{\diag}{d}
\DeclareMathOperator{\prim}{Prim}
\DeclareMathOperator{\mprim}{m-Prim}
\DeclareMathOperator{\rep}{Rep_{2}}
\DeclareMathOperator{\ev}{ev}
\DeclareMathOperator{\aut}{Aut}
\newcommand{\cai} {\mathbf{cai}}

\DeclareMathOperator{\hull}{hull}
\DeclareMathOperator{\Span}{Span}

\DeclarePairedDelimiterX\braket[2]{\langle}{\rangle}{#1 \delimsize\vert #2}

\begin{document}

\title[Rigidity for Isomorphisms]{Rigidity for Isomorphisms between Operator Algebras with Commutative Diagonals}

\author{Elias G. Katsoulis}
\address{Department of Mathematics, East Carolina University, Greenville, NC 27858, USA.}
\email{katsoulise@ecu.edu}

\author{Feifei Miao}
\address{School of Mathematics and Information Science, Shandong Technology and Business University, Yantai, Shandong, 264005, China; Yantai Key Laboratory of Big Data Modeling and Intelligent Computing, Yantai, Shandong, 264005}
\email{202314114@sdtbu.edu.cn}

\author{Wenming Wu}
\address{School of Mathematical Sciences, Chongqing Normal University, Chongqing, 401331, China.}
\email{wuwm@amss.ac.cn}

\author{Wei Yuan}
\address{Institute of Mathematics, Academy of Mathematics and Systems Science, Chinese Academy of Sciences, Beijing, 100190, China; School of Mathematical Sciences, University of Chinese Academy of Sciences, Beijing 100049, China.}
\email{wyuan@math.ac.cn}

\date{}

\begin{abstract}
We show that two families of operator algebras, the CSL algebras of multiplicity free CSLs and the semicrossed products of commutative C$^*$-algebras, demonstrate a strong form of rigidity with respect to isometric isomorphisms. Specifically, the isomorphism class of any such algebra remains unchanged within its family, even if we allow for isomorphism after tensoring with operator algebras containing the compact operators. For semicrossed products of commutative C$^*$-algebras, the same conclusion holds even when tensoring with operator algebras whose diagonals are irreducibly acting. Collectively, these results imply rigidity with respect to stable isomorphisms: two algebras are isometrically isomorphic if and only if they are stably isomorphic. 
\end{abstract}

\subjclass[2020]{47L40, 47L55, 47L65, 46L05}
\keywords{stable isomorphism, CSL algebra, semicrossed product} 

\maketitle

\section{Introduction}
Motivated by the work of Dor-On, Eilers, and Geffen \cite{DEG20}, as well as by the classification problem for non-selfadjoint crossed products, Kakariadis, Katsoulis and Li \cite{KaKaLi23} recently investigated the stable isomorphisms of operator algebras. Among other results, they showed that if two such algebras have diagonals that are either $c_0$-isomorphic or satisfy cancellation and have $K_0$-groups isomorphic to $\mathbb{Z}$, then stable isomorphism implies isometric isomorphism of the original algebras. In this paper, we continue the investigation of \cite{KaKaLi23} by identifying additional families of operator algebras for which any two stably isomorphic algebras within the same family are, in fact, isometrically isomorphic--and, as we shall see, even stronger conclusions can be drawn. 

To set the stage for our results, we first establish some notation. Let $\Clk$ denote the algebra of compact operators on a separable Hilbert space $\CK$ (for instance, $\CK = l^2(\Nb)$). For an operator algebra $\CA \subset \Clb(\CH)$ acting on a Hilbert space $\CH$, its diagonal is defined by 
\begin{align*}
 \diag(\CA) :=  \CA \cap \CA^*.
\end{align*}
For the rest of the paper we make the blanket assumption that \textit{all operator algebras necessarily contain in their diagonal a contractive approximate identity for the algebra.}
The stabilization of an operator algebra $\CA$ is the minimal tensor product
\begin{align*}
  \CA \otimes_{\min} \Clk \subset \Clb(\CH \otimes \CK).  
\end{align*}
The definition of the C$^*$-algebra $\diag(\CA)$ and the operator algebra $\CA \otimes_{\min} \Clk$ are independent of the particular representation of $\CA$ (see, for example, Section~2.1 and 2.2 of \cite{BlMe04}). In what follows, we omit the subscript 'min' and write $\CA \otimes_{\min} \Clk$ simply as $\CA \otimes \Clk$, since no other tensor products of operator algebras will be considered. 

The notion of stable isomorphism plays a crucial role in the K-theory of operator algebras. Two operator algebras are said to be \textit{stably (completely) isomorphic} if their stabilization are (completely) isometrically isomorphic. For $\sigma$-unital C$^*$-algebras, a fundamental theorem due to Brown, Green, and Rieffel \cite{BGR77} asserts that stable isomorphism coincides with strong Morita equivalence. This result has since been extended to the realm of non-selfadjoint operator algebras and operator spaces (see \cite{GE16, EK17, GE19} and the reference therein). 

In general, stable isomorphism or various forms of Morita equivalence (e.g, strong/weak Morita equivalence, strong/weak $\Delta$-equivalence) define equivalence relations that are coarser than (completely) isometric isomorphism of operator algebras, even for C$^*$-algebras. However, when restricted to operator algebras with commutative diagonals, the situation simplifies considerably. Indeed, stably isomorphic commutative C$^*$-algebras are necessarily isomorphic. This observation suggests that, for broad classes of operator algebras with commutative diagonals\footnote{Perhaps even for all such algebras.}, allowing stable isomorphisms in place of (isometric) isomorphisms does not enlarge the underlying isomorphism class. As shown in \cite{KaKaLi23}, this holds not only for commutative C$^*$-algebras but also, under mild conditions, for two classes of non-selfadjoint operator algebras whose commutative diagonals are either $c_0$-isomorphic or contractible as topological spaces. 

To develop further these results, the following sections provide additional evidence supporting this assertion. To make the paper as self-contained as possible, we provide the necessary background for the proofs in Section~2 and 3, while relying on \cite{BlMe04}, \cite{GM90}, and \cite{KR97} as our primary references for operator algebras, C$^*$-algebras, and von Neumann algebras, respectively. Throughout the paper, we adopt the convention that uppercase letters denote operators acting on Hilbert spaces, which appear mainly in Section~2, whereas lowercase letters denote abstract operators, which appear mainly in Section~3. 

We begin in Section~2 by studying isomorphisms between CSL algebras. These algebras were introduced in the 1970s by Arveson \cite{Arv2} in a paper that attracted a lot of attention and a strong following. As a well-studied class of operator algebras (see \cite{DavNest} and the references therein), CSL algebras serve a natural candidate for examining the type of rigidity that we are interested in this paper. Since we are concerned only with algebras having commutative diagonals, we restrict our attention to CSL algebras of multiplicity-free CSLs. First, we show that any isometric isomorphism between the stabilizations of such CSL algebras induces a $*$-isomorphism between certain maximal abelian self-adjoint subalgebras in the multiplier algebras of the stabilizations. Then, by adapting an argument similar to that used in the proof of Ringrose’s isomorphism theorem for nest algebras, we deduce in Corollary~\ref{cor:CSL_case} that stable isomorphism of such CSL algebras is equivalent to unitary equivalence of the underlying CSLs. We also note that Eleftherakis and Paulsen \cite{EP08} proved that for two CSL algebras $\CA$ and $\CB$, the algebras $(\CA \otimes \Clb(\CK))^{\wot}$ and $(\CB \otimes \Clb(\CK))^{\wot}$ are completely isometrically isomorphic if and only if there exists a lattice isomorphism between their lattices satisfying a certain continuity condition. Their result is based on earlier work on the Morita theory of reflexive algebras (\cite{GE10}; see also \cite{GE08}). Although our result is proved without using the TRO machinery as in \cite{EP08}, it would be interesting to see whether Corollary~\ref{cor:CSL_case} can also be deduced using TRO techniques.

In Section~3, we shift our attention to semicrossed products by $\Zb$. Such semicrossed products for C$^*$-dynamical systems were introduced by Peters \cite{JP84}, following earlier work of Arveson \cite{Arv1} and Arveson and Josephson \cite{AJ69}, and have been studied extensively over the last 30 years \cite{DKak, DK1, DK3, DK08, HH, KRSigma, Pow1}. The recent monograph of Davidson \cite{Davlast} provides a comprehensive exposition of the core results in the area. For more general semicrossed products, the first systematic study was carried out in \cite{KakK1}, where automorphisms of arbitrary unital operator algebras were considered. As in Section~2, we restrict our attention to isomorphisms of semicrossed products by $\Zb^+$ of abelian C$^*$-algebras since we focus on operator algebras with commutative diagonals. Let $\alpha$ be a $*$-automorphism of a commutative C$^*$-algebra $\CA$. By Theorem 2.12 in \cite{KR19}, there is a canonical isomorphism
\begin{align*}
    \left (\CA \otimes \Clk^{+} \right) \rtimes_{\alpha \otimes \Ad \lambda} \Zb \simeq \left (\CA \rtimes_{\alpha} \Zb^{+} \right) \otimes \Clk, 
\end{align*}
where $\Clk^{+}$ denotes the algebra of upper-triangular compact operators with respect to a $\Zb$-ordered orthonormal basis, $\lambda$ is the left regular representation of $\Zb$, and $\CA \rtimes_{\alpha} \Zb^{+}$ is the semicrossed product as defined in \cite{JP84}. Thus, the classification problem for non-self-adjoint crossed product algebras of the form
\begin{align*}
    \left (\CA \otimes \Clk^{+} \right) \rtimes_{\alpha \otimes \Ad \lambda} \Zb
\end{align*}
reduces to the classification of semicrossed products up to stable isomorphism. In \cite{KaKaLi23} it was shown that for homeomorphisms $\sigma_i$ on a contractible, compact Hausdorff space $X$, the following statements are equivalent:
\begin{itemize}
    \item[(i)] $\sigma_1$ and $\sigma_2$ are conjugate;
    \item[(ii)] $C(X) \rtimes_{\sigma_i} \Zb^{+}$, $i = 1, 2$, are (completely) isometrically isomorphic;
    \item[(iii)] $\left (C(X) \rtimes_{\sigma_i} \Zb^{+} \right) \otimes \Clk$, $i = 1, 2$, are (completely) isometrically isomorphic;
\end{itemize}
By amplifying the triangular matrix representations used in \cite{DK08, DK11}, we remove the assumption that $X$ is contractible and establish the same equivalence under the sole condition that the maps $\sigma_i$, $i= 1, 2$, are proper continuous. In particular, this resolves the classification problem for non-selfadjoint crossed products of the form $   \left (\CA \otimes \Clk^{+} \right) \rtimes_{\alpha \otimes \Ad \lambda} \Zb$, where $\CA$ is commutative and $\alpha$ is a $*$-automorphism.

As mentioned earlier, our results go well beyond stable isomorphisms. More precisely, Theorem~\ref{thm:stabl_to_unitary} implies that if $\CA_i$, $i=1,2$, are CSL algebras of multiplicity-free CSLs, and $\CB_i$, $i=1,2$, are (concrete) operator algebras containing the compact operators, then the algebras $\CA_i\otimes \CB_i$, $i=1,2$, are isometrically isomorphic if and only if the algebras $\CA_i$, $i=1,2$, are. A similar result follows from Theorem~\ref{mainthm;semicr}, where this time $\CA_i$, $i=1,2$, are semicrossed products of commutative C$^*$-algebras, and $\CB_i$, $i=1,2$, are operator algebras with irreducibly acting diagonals, e.g., simple C$^*$-algebras. These results show that the isomorphisms between algebras with commutative diagonals exhibit highly rigid behavior, which we believe deserves further investigation. We intend to explore  this behavior further in future work.

\section{Stable isomorphisms of CSL algebras}

For a set of bounded operators $\CS$ acting on a Hilbert space $\CH$, define
\begin{align*}
    \Lat(\CS) := \{\mbox{$P$ is a projection in $\Clb(\CH)$}: (I-P)TP = 0, \forall T \in \CS\},
\end{align*}
which is a subspace lattice consisting of all projections whose ranges are invariant under every operator in $\CS$. More generally, a subspace lattice is a strongly closed collection of projections on a Hilbert space that is closed under the lattice operations $\wedge$ and $\vee$, and contains both the zero projection $0$ and the identity operator $I$. A commutative subspace lattice (abbr. CSL) is a subspace lattice whose projections all commute.

For a subspace lattice $\CL \subset \Clb(\CH)$, the bounded operators that leave invariant the range of every projection in $\CL$ form an algebra, denoted by $\Alg(\CL)$, whose diagonal coincides with the commutant  
\begin{align*}
    \CL' := \{T \in \Clb(\CH): TP = PT, \forall P \in \CL\} 
\end{align*}
of $\CL$. In particular, if $\CL$ is a CSL, the commutativity of $\diag(\Alg(\CL))$ is equivalent to $\CL$ being multiplicity free, i.e., the von Neumann algebra $\CL''$ generated by the projections in $\CL$ is a maximal abelian self-adjoint algebra (MASA).

In the remainder of this section, we show that two CSL algebras $\Alg(\CL_1)$ and $\Alg(\CL_2)$ with commutative diagonals are isometrically isomorphic if and only if their stabilizations $\Alg(\CL_1) \otimes \Clk$ and $\Alg(\CL_2) \otimes \Clk$ are isometrically isomorphic. In fact, we establish a more general result and derive the above fact as a corollary. In this process, we need to consider the multiplier algebra of $\Alg(\CL_i)$ tensored with an approximately unital operator algebra; therefore, we begin by briefly recalling some basic facts about multiplier algebras of operator algebras.\\

An operator algebra is called \textit{approximately unital} if it possesses a contractive approximate identity ($\cai$). As we mentioned in the introduction, all operator algebras appearing in this paper are assumed approximate unital with a cai contained in their diagonal. Let $\CA$ be an operator algebra. The double dual $\CA^{**}$ of $\CA$, equipped with Arens product, is an operator algebra that admits a norm-one unit, and the canonical isometry $\CA \to \CA^{**}$ is an algebra homomorphism. The multiplier algebra $\Clm(\CA)$ of $\CA$ is the idealizer of $\CA$ in $\CA^{**}$, that is, 
\begin{align*}
    \Clm(\CA) = \{t \in \CA^{**}: t \CA \subset \CA, \CA t \subset \CA\}.
\end{align*}
Moreover, if $\CA \subset \Clb(\CH)$ is a concrete approximately unital operator algebra acting non-degenerately on a Hilbert space $\CH$, then $\Clm(\CA)$ can be identified with the idealizer of $\CA$ in $\Clb(\CH)$:
\begin{align}\label{equ:mul_concrete}
    \{T \in \Clb(\CH): T \CA \subset \CA, \CA T \subset \CA\}
\end{align}
(see Section 2.6 in \cite{BlMe04}).\newline

The following facts about multiplier algebras will be used later. 

\begin{prop}\label{prop:ext_to_double_dual}
    Let $\CA$ and $\CB$ be operator algebras and $\varphi: \CA \to \Clm(\CB)$ be a contractive homomorphism. Then there exists a unique $w^*$-continuous contractive homomorphism $\overline{\varphi}: \CA^{**} \to \CB^{**}$ extending $\varphi$ with $\|\overline{\varphi}\| = \|\varphi\|$. In particular, if $\CB = \Clk$, i.e., $\Clm(\CB) = \Clb(\CK)$, then $\varphi$ extends to a representation of $\CA^{**}$ on the Hilbert space $\CK$. In this case, $\varphi$ is non-degenerate if and only if $\overline{\varphi}(I) = I$. 
\end{prop}

\begin{proof}
By composing $\varphi$ with the canonical isometry $\Clm(\CB) \to \CB^{**}$, we can view $\varphi$ as a bounded homomorphism from $\CA$ into $\CB^{**}$. Hence there exists a unique $w^*$-continuous extension $\overline{\varphi} : \CA^{**} \to \CB^{**}$ with $\|\overline{\varphi}\| = \|\varphi\|\leq 1$ by Lemma A.2.2 in \cite{BlMe04}. Since the Arens products on $\CA^{**}$ and $\CB^{**}$ are separately $w^*$-continuous, Goldstine’s Theorem implies that $\overline{\varphi}$ is an algebra homomorphism. Finally, note that $\Clk^{**} = \CB(\CK)$, the last statement follows.
\end{proof}

The following result appears in \cite[Proposition 2.6.12]{BlMe04} but with the stronger assumption that the homomorphism $\phi$ be completely contractive. As we need $\phi$ to be merely contractive, we provide a proof for completeness.

\begin{prop}\label{prop:ext_to_mul}
Let $\CA$ and $\CB$ be operator algebras, and let $\varphi: \CA \to \Clm(\CB)$ be a multiplier-nondegenerate contractive homomorphism, i.e., for every $\cai$ $\{e_i\}$ of $\CA$,  
\begin{align}\label{equ:prop:ext_to_mul_I}
    \lim_{i} \varphi(e_i) b = b = \lim_{i} b\varphi(e_i), \quad \forall b \in \CB. 
\end{align}
Then $\varphi$ extends uniquely to a unital homomorphism $\overline{\varphi}: \Clm(\CA) \to \Clm(\CB)$ with $\|\overline{\varphi}\| = \|\varphi\|$. Moreover, $\overline{\varphi}$ is an isometry whenever $\varphi$ is.
\end{prop}

\begin{proof}
    By Proposition~\ref{prop:ext_to_double_dual}, there exists a $w^*$-continuous homomorphism $\overline{\varphi}: \CA^{**} \to \CB^{**}$ extending $\varphi$ with $\|\overline{\varphi}\| = \|\varphi\|$. For every $t \in \Clm(\CA)$, note that $\varphi(te_i), \varphi(e_i t) \in \Clm(\CB)$. Then $\overline{\varphi}(t) \in \Clm(\CB)$,  since for every $b \in \CB$, 
\begin{align*}
    &\overline{\varphi}(t) b = \lim_{i} \overline{\varphi}(t) \varphi(e_i) b = \lim_{i} \varphi(t e_i) b \in \CB \mbox{, and }\\
    &b\overline{\varphi}(t) = \lim_{i} b \varphi(e_i) \overline{\varphi}(t)   = \lim_{i} b \varphi(e_i t)  \in \CB.
\end{align*}
    To show that $\overline{\varphi}$ is unital, note that $e_i \to I$ in the $w^*$-topology by Proposition 2.5.8 in \cite{BlMe04}. Equation~\eqref{equ:prop:ext_to_mul_I} then reads $\overline{\varphi}(I)b = b = b \overline{\varphi}(I)$. It follows from Goldstine’s Theorem and the separate $w^*$-continuity of the Arens product that $\overline{\varphi}(I) = I$. The uniqueness of the extension then follows from the observation that, for any extension $\rho: \Clm(\CA) \to \Clm(\CB)$ of $\varphi$, we have $\varphi(te_i) = \rho(t) \varphi(e_i)$, which converges in the $w^*$-topology to both $\overline{\varphi}(t)$ and $\rho(t)$.

If $\varphi$ is an isometry, then, using the description of the multiplier algebra in Theorem 2.6.2 in \cite{BlMe04}, for $t \in \Clm(\CA)$ we have 
\begin{align*}
    \|\overline{\varphi}(t)\| = \sup_{a \in \CA_1, b \in \CB_1} \|\overline{\varphi}(t)\varphi(a)b\|
    = \sup_{a \in \CA_1, b \in \CB_1} \|\varphi(ta)b\| = \sup_{a \in \CA_1} \|\varphi(ta)\| = \sup_{a \in \CA_1} \|ta\| = \|t\|, 
\end{align*}
where $\CA_1$ and $\CB_1$ denote the unit balls of $\CA$ and $\CB$, respectively. Hence, $\overline{\varphi}$ is an isometry.   
\end{proof}

\begin{prop}\label{prop:mdiag_eq_diagm}
    If $\CA$ is an operator algebra, then 
    \begin{align*}
        \diag (\Clm(\CA)) = \Clm (\diag(\CA)). 
    \end{align*} 
\end{prop}

\begin{proof}
    Without loss of generality, we may assume that $\CA$ is represented as a non-degenerate subalgebra of $\Clb(\CH)$. Let $T \in \diag (\Clm(\CA))$. Since 
\begin{align*}
    \{TA, A^*T^*, AT, T^*A^* \} \subset \CA, \quad \forall A \in \diag(\CA),
\end{align*}
it follows that $T \in \Clm(\diag(\CA))$. Hence, $\diag (\Clm(\CA)) \subseteq \Clm (\diag(\CA))$. Conversely, for $S \in \Clm(\diag(\CA))$, we have
    \begin{align*}
        SA &= \lim_{i} (S E_{i}) A \in \CA, \quad AS = \lim_{i}  A(E_{i} S)\in \CA
    \end{align*}
for every $A \in \CA$, where $\{E_i \}$ is a $\cai$ of $\CA$ contained in $\diag(\CA)$. Therefore, $\diag (\Clm(\CA)) = \Clm (\diag(\CA))$. 
\end{proof}

\begin{prop}\label{prop:diag_ideal_to_ideal_mul}
    Let $\CA$ be a closed ideal of an operator algebra $\CB$. Then $\diag(\CA)$ is an closed ideal in $\diag \left (\Clm(\CB) \right)$. 
\end{prop}

\begin{proof}
Without loss of generality, we may assume that $\CB$ is represented as a non-degenerate subalgebra of $\CB(\CH)$. Since
\begin{align*}
    \{AB, BA, A^*B^*, B^*A^*\} \subset \CA, \quad \forall A \in \CA^* \cap \CA, B \in \CB^* \cap \CB,
\end{align*}
    it follows that $\diag(\CA)$ is a closed ideal of $\diag(\CB)$. Note that $\diag(\CB)$ is a closed ideal in $\diag\left (\Clm(\CB) \right) = \Clm\left(\diag(\CB) \right)$ by Proposition \ref{prop:mdiag_eq_diagm}. Hence, $\diag(\CA)$ is a closed ideal in $\diag\left (\Clm(\CB) \right)$ (see Remark 3.1.2 in \cite{GM90}). 
\end{proof}


The next lemma follows immediately from the fact that the multiplier algebra of a discrete operator algebra is contained in its double commutant (see [2.6.5] in \cite{BlMe04}).

\begin{lem}\label{lem:ab_mul}
Let $\CA$ be a MASA in $\Clb(\CH)$, and let $\CD \subset \Clk$ be a MASA generated by a complete family of mutually orthogonal rank-one projections. Then the commutant $(\CA \otimes \CD)'$ is a MASA in $\Clb(\CH \otimes \CK)$ and one has
\begin{align*}
 \Clm(\CA \otimes \CD) = (\CA \otimes \CD)'.   
\end{align*}
\end{lem}

\begin{lem}\label{lem:center_of_mul}
Let $\CL$ be a CSL on a Hilbert space $\CH$, and let $\CB \subset \Clb(\CK)$ be an approximately unital operator algebra containing $\Clk$. Then 
\begin{align*}
    \diag \left(\Clm(\Alg (\CL) \otimes \CB) \right) ' = \CL'' \otimes I.
\end{align*}
In particular the center of $\diag \left(\Clm(\Alg(\CL) \otimes \CB) \right)$ is $\CL'' \otimes I$.
\end{lem}

\begin{proof}
First, note that $(I \otimes \diag(\CB))' = \Clb(\CH) \otimes I$, since $\Clk \subset \diag(\CB)$. From the inclusion
\begin{align*}
    \CL' \otimes \diag(\CB) \subset \diag\left(\Clm(\Alg(\CL) \otimes \CB) \right) \subset \left (\CL' \otimes \Clb(\CK) \right)^{-\wot},
\end{align*}
it follows that $\CL'' \otimes I$ is the commutant of $\diag \left (\Clm(\Alg(\CL) \otimes \CB) \right)$, and hence also its center.
\end{proof}

In what follows, $\CL_i$ denotes a CSL on a Hilbert space $\CH_i$, while $\CB_i \subset \Clb(\CK)$ denotes an operator algebra that contains the compact operators $\Clk$, $i = 1, 2$.

\begin{lem}\label{lem:non_ret}
Suppose $\varphi: \Alg(\CL_1) \otimes \CB_1 \to \Alg(\CL_2) \otimes \CB_2$ is an isometric isomorphism. Then the restriction of $\varphi$ to $\CL_1' \otimes \Clk$ is non-degenerate.
\end{lem}

\begin{proof}
    Let $\overline{\varphi}: \Clm (\Alg(\CL_1) \otimes \CB_1) \to \Clm (\Alg(\CL_2) \otimes \CB_2)$ denote the unital isometric isomorphism extending $\varphi$. Then $\overline{\varphi}$ restricts to a $*$-isomorphism between the diagonals of $\Clm(\Alg(\CL_1) \otimes \CB_1)$ and $\Clm(\Alg(\CL_2) \otimes \CB_2)$. Note that $\Alg(\CL_1) \otimes \Clk$ is an closed ideal in $\Alg(\CL_1) \otimes \CB_1$. By Proposition~\ref{prop:diag_ideal_to_ideal_mul},
\begin{align*}
    \CL_1' \otimes \Clk = \diag(\Alg(\CL_1) \otimes \Clk) 
\end{align*}
    is a closed ideal in $\diag \left ( \Clm (\Alg(\CL_1) \otimes \CB) \right)$ (see also Lemma 2.2 in \cite{KaKaLi23}). Thus, $\overline{\varphi(\CL_1' \otimes \Clk) (\CH_2 \otimes \CK)}$ is an invariant subspace of
\begin{align*}
    \diag \left ( \Clm (\Alg(\CL_2) \otimes \CB_2) \right) = \overline{\varphi}(\diag \left (\Clm (\Alg(\CL_1) \otimes \CB_1) \right)).
\end{align*}
Let $Q$ be the projection onto the subspace $\overline{\varphi(\CL_1' \otimes \Clk) (\CH_2 \otimes \CK)}$. By Lemma \ref{lem:center_of_mul}, we have $Q \in \CL_2'' \otimes I$. Since $\overline{\varphi}$ induces a $*$-isomorphism between the centers of $\diag \left (\Clm(\Alg(\CL_i) \otimes \CB_i) \right)$, $i = 1, 2$, there exists a projection $P \in \CL_1''$ such that $Q = \overline{\varphi}(P \otimes I)$. Then 
\begin{align*}
    \varphi \left( (\CL_1'(I-P)) \otimes \Clk \right) = \varphi(\CL_1' \otimes \Clk) \overline{\varphi}\left ( (I-P) \otimes I \right) = 0. 
\end{align*}
Since $\varphi$ is injective, we conclude that $P = I$, and hence the restriction of $\varphi$ to $\CL_1' \otimes \Clk$ is non-degenerate.
\end{proof}

\begin{lem}\label{lem:unitary_on_diag}
Suppose $\CL_1$ and $\CL_2$ are multiplicity free. Let $\CD \subset \Clk$ be a MASA generated by a complete family of mutually orthogonal rank-one projections. If 
\begin{align*}
 \varphi: \Alg(\CL_1) \otimes \CB_1 \to \Alg(\CL_2) \otimes \CB_2   
\end{align*}
is an isometric isomorphism, then there exists a unitary $U: \CH_2 \otimes \CK \to \CH_1 \otimes \CK$ such that 
\begin{align*}
    \overline{\varphi}(A) = U^* AU, \quad \forall A \in \Clm(\CL_1' \otimes \CD) \cap \diag \left (\Clm(\Alg(\CL_1) \otimes \CB_1) \right), 
\end{align*}
where $\overline{\varphi}: \Clm(\Alg(\CL_1) \otimes \CB_1) \to \Clm(\Alg(\CL_2) \otimes \CB_2)$ is the unique isometric extension of $\varphi$.
\end{lem}

\begin{proof}
We use $\CA$ to denote the C$^*$-algebra 
\begin{align*}
    \Clm(\CL_1' \otimes \CD) \cap \diag \left ( \Clm(\Alg(\CL_1) \otimes \CB_1) \right ).
\end{align*}
By Proposition~\ref{lem:non_ret}, the restriction of $\varphi$ to $\CL_1' \otimes \CD$ is non-degenerate, since $\CL_1' \otimes \CD$ contains an approximate unit of $\CL_1' \otimes \Clk$. Let 
    \begin{align*}
        \phi: \Clm(\CL_1' \otimes \CD) \to \Clb(\CH_2 \otimes \CK)   
    \end{align*}
    denote the unique extension of $\varphi|_{\CL_1' \otimes \CD}$. Then $\phi$ is an injective $*$-homomorphism agreeing with $\overline{\varphi}$ on $\CA$, as $\CL_1' \otimes \CD$ is a closed ideal in $\CA$ (see Proposition~\ref{prop:ext_to_mul} or Proposition 2.1 in \cite{EL95}). In particular, by Lemma~\ref{lem:center_of_mul}, 
\begin{align*}
    \phi(\CL_1' \otimes I) = \overline{\varphi}(\CL_1' \otimes I) = \CL_2' \otimes I, 
\end{align*}
since $\overline{\varphi}$ induces a $*$-isomrophsim between the centers of $\diag \left ( \Clm(\Alg(\CL_1) \otimes \CB_1) \right )$ and $\diag \left ( \Clm(\Alg(\CL_2) \otimes \CB_2) \right )$.

If we can show that $\phi(\Clm(\CL_1' \otimes \CD))$ is a MASA in $\Clb(\CH_2 \otimes \CK)$, then the lemma follows from Corollary 10.15 in \cite{SZ19}, since every MASA is standard. We now proceed to show that $\phi(\Clm(\CL_1' \otimes \CD))$ indeed forms a MASA in $\Clb(\CH_2 \otimes \CK)$. Let $\{E_i\}_{i=1}^{\infty}$ be the family of rank-one projections generates $\CD$. Note that
\begin{align*}
    \CL_i' \otimes \Clk  \subset \diag \left ( \Clm(\Alg(\CL_i) \otimes \CB_i) \right ) \subset \left ( \CL_i' \otimes \Clb(\CK) \right )^{-\wot}.
\end{align*}
Hence, 
\begin{align*}
    &\phi(I \otimes E_i) \diag \left ( \Clm(\Alg(\CL_2) \otimes \CB_2) \right ) \phi(I \otimes E_i)\\
    =& \overline{\varphi}\left ( (I \otimes E_i) \diag \left ( \Clm(\Alg(\CL_1) \otimes \CB_1) \right ) (I \otimes E_i) \right ) \\
    =& \overline{\varphi}\left ( \CL_1' \otimes E_i \right ) = (\CL_2' \otimes I) \phi(I \otimes E_i).
\end{align*}
Since $\CL_2' \otimes \Clk$ is WOT-dense in $\left ( \CL_2' \otimes \Clb(\CK) \right )^{-\wot}$ and $(\CL_2' \otimes I) \varphi(I \otimes E_i)$ is a von Neumann algebra (see, for example, Proposition 5.5.6 in \cite{KR97}), we have 
\begin{align*}
    \phi(I \otimes E_i) \left (\CL_2' \otimes \Clb(\CK) \right )^{-\wot} \phi(I \otimes E_i) = (\CL_2' \otimes I) \phi(I \otimes E_i),
\end{align*}
i.e., $\phi(I \otimes E_i)$ is an abelian projection in $\left ( \CL_2' \otimes \Clb(\CK) \right )^{-\wot}$. Then, for every $T \in \phi(\Clm(\CL_1' \otimes \CD))' \subset \left (\CL_2' \otimes \Clb(\CK) \right )^{-\wot}$, there exists an operator $A_i \in \CL_1'$ such that
    \begin{align*}
        T \phi(I \otimes E_i) =  \phi(A_i \otimes E_i).
    \end{align*}
By Lemma~\ref{lem:ab_mul}, $\sum_{i=1} A_i \otimes E_i \in \Clm(\CL_1' \otimes \CD)$. Then
\begin{align*}
    T = \phi \left (\sum_{i=1} A_i \otimes E_i \right) \in \phi(\Clm(\CL_1' \otimes \CD)), 
\end{align*}
since $\sum_{i=1}^n \phi(I \otimes E_i)$ converges to the identity operator in $\Clb(\CH_2 \otimes \CK)$ in the SOT. Therefore, $\phi(\Clm(\CL_1' \otimes \CD))$ is a MASA in $\CB(\CH_2 \otimes \CK)$. 
\end{proof}

\begin{thm}\label{thm:stabl_to_unitary}
Suppose $\CL_1$ and $\CL_2$ are multiplicity free. If $\Alg(\CL_1) \otimes \CB_1$ and $\Alg(\CL_2) \otimes \CB_2$ are isometrically isomorphic, then $\Alg(\CL_1)$ and $\Alg(\CL_2)$ are unitarily equivalent.
\end{thm}

\begin{proof}
Note that $\Lat\Alg(\CL_i)$ is a reflexive lattice contained in $\CL_i'$. Therefore, upon replacing $\CL_i$ with $\Lat\Alg(\CL_i) (\subset \CL_i')$, we may, without loss of generality, assume that $\CL_i$ is reflexive, for $i = 1, 2$.

Suppose $\varphi: \Alg(\CL_1) \otimes \CB_1 \to \Alg(\CL_2) \otimes \CB_2$ is an isometric isomorphism, and let
\begin{align*}
    \overline{\varphi}: \Clm(\Alg(\CL_1) \otimes \CB_1) \to \Clm(\Alg(\CL_2) \otimes \CB_2)
\end{align*}
denote its isometric extentsion. Let $\CD \subset \Clk$ be a MASA generated by a maximal family of mutually orthogonal rank-one projections. By Lemma~\ref{lem:unitary_on_diag}, there exists a unitary $U: \CH_2 \otimes \CK \to \CH_1 \otimes \CK$ such that 
\begin{align*}
    \overline{\varphi}(A) = U^* AU, \quad \forall A \in \Clm(\CL_1' \otimes \CD) \cap \diag \left (\Clm(\Alg(\CL_1) \otimes \CB_1) \right). 
\end{align*}
Then the composition
\begin{align*}
    \Phi := \left ( \Clm( \Alg(\CL_1) \otimes \CB_1) \xrightarrow{ \overline{\varphi}} \Clm( \Alg(\CL_2) \otimes \CB_2) \xrightarrow{U (-) U^*} \Clm( U (\Alg(\CL_2) \otimes \CB_2)U^*) \right) 
\end{align*}
is an isometric isomorphism such that
\begin{align*}
    \Phi(A) = A, \quad \forall A \in \Clm(\CL_1' \otimes \CD) \cap \diag \left (\Clm(\Alg(\CL_1) \otimes \CB_1) \right).  
\end{align*}

Since $\Clm(\Alg(\CL_i) \otimes \CB_i)$ is $\Wot$-dense in $\Alg(\CL_i \otimes I) = \left (\Alg(\CL_i) \otimes \Clb(\CK) \right)^{-\wot}$, we have
\begin{align*}
    [I - (P \otimes I)] \Phi(T) (P\otimes I) = \Phi( [I - (P \otimes I)] T (P \otimes I)) = 0
\end{align*}
for every $T \in \Clm(\Alg(\CL_1) \otimes \CB_1)$ and $P \in \CL_1$. Hence,  
\begin{align*}
    \Alg(U(\CL_2 \otimes I)U^*) \subseteq \Alg(\CL_1 \otimes I)
\end{align*}
since $\Clm(U(\Alg(\CL_2) \otimes \CB_2)U^*)$ is $\Wot$-dense in $\Alg(U(\CL_2 \otimes I)U^*)$. Note that
\begin{align*}
    U(Q \otimes I)U^* \in \CL_1' \otimes I, \quad \forall Q \in \CL_2. 
\end{align*}
We have
\begin{align*}
    [I - U(Q \otimes I)U^*] T (U(Q \otimes I)U^*) = \Phi^{-1}\left ( [I - U(Q \otimes I)U^*] \Phi(T)  (U(Q \otimes I)U^*) \right ) = 0 
\end{align*}
for every $T \in \Clm(\Alg(\CL_1) \otimes \CB_1)$. Therefore, $\Alg(U(\CL_2 \otimes I)U^*) = \Alg(\CL_1 \otimes I)$. Since the lattices $U(\CL_2 \otimes I)U^*$ and $\CL_1 \otimes I$ are both reflexive, it follows that $\CL_2 \otimes I = U^*(\CL_1 \otimes I)U$. 

Since the abelian von Neumann algebra $\CL_i' \otimes I$ is generated by the projections in $\CL_i \otimes I$, the map $P \otimes I  \mapsto U^*(P \otimes I)U$ defines a $*$-isomorphism from $\CL_1' \otimes I$ to $\CL_2' \otimes I$. Let $V: \CH_2 \to \CH_1$ be the unitary that implements the following composition of $*$-isomorphisms:
\begin{align*}
    \CL_1' \to \CL_1' \otimes I \xrightarrow{U^*(-)U} \CL_2' \otimes I \to \CL_2', 
\end{align*}
where the first and last arrows are induced by the canonical $*$-isomorphisms $\CL_i' \simeq \CL_i' \otimes I$, for $i=1,2$ (see Corollary 10.15 in \cite{SZ19}). It is clear that $\Alg(\CL_2) = V^* \Alg(\CL_1) V$. 
\end{proof}

The following corollary follows immediately from Theorem~\ref{thm:stabl_to_unitary}.

\begin{cor}\label{cor:CSL_case}
Let $\CL_i$ be a multiplicity free CSL on a Hilbert space $\CH_i$, $i = 1, 2$. Then the following statements are equivalent:
    \begin{enumerate}
        \item[(i)] $\Alg(\CL_1)$ and $\Alg(\CL_2)$ are unitarily equivalent.;
        \item[(ii)] $\Alg(\CL_1)$ and $\Alg(\CL_2)$ are (completely) isometrically isomorphic;
        \item[(iii)] $\Alg(\CL_1) \otimes \Clk$ and $\Alg(\CL_2) \otimes \Clk$ are (completely) isometrically isomorphic.
    \end{enumerate}
\end{cor}

\begin{question}
In light of the preceding results, it is natural to ask whether the same assertion remains valid for reflexive algebras associated with reflexive lattices possessing commutative diagonals--for instance, those determined by lattices containing a CSL. One simple way to construct such a lattice is to adjoin a rank-one projection to a CSL; a special case of this construction is considered in \cite{Hou10}.
\end{question}

\section{Stable isomorphisms of semicrossed products of C$^*$-algebras by $\Zb^{+}$}

In this paper, a dynamical system is a pair $(\CA, \alpha)$, where $\CA$ is an approximately unital operator algebra and $\alpha: \CA \to \CA$ is a completely contractive endomorphism. In the case where $\CA$ is a C$^*$-algebra, the pair  $(\CA, \alpha)$ is said to be a C$^*$-dynamical system. The \textit{covariance algebra} $\CA(\alpha)$ is the linear space spanned by the elements of the form $\mathbf{s}^k a$, with $a \in \CA$, $k \geq 0$, endowed with the multiplication specified by the \textit{covariance relation}
\begin{align*}
    a \mathbf{s} = \mathbf{s} \alpha(a), \quad \forall a \in \CA. 
\end{align*}
An \textit{isometric covariant representation} of $(\CA, \alpha)$ on a Hilbert space $\CH$ is a pair $(\pi, V)$, where $\pi: \CA \to \Clb(\CH)$ is a completely contractive homomorphism and $V \in \Clb(\CH)$ is an isometry satisfying 
\begin{align*}
    \pi(a) V = V \pi(\alpha(a)), \quad \forall a \in \CA. 
\end{align*}
Such a pair $(\pi, V)$ extends uniquely to a representation of the covariance algebra $\CA(\alpha)$ on $\CH$, denoted by $\pi \rtimes_\alpha V$. 

For a dynamical system $(\CA, \alpha)$, the (isometric) semicrossed product $\CA \rtimes_{\alpha} \Zb^{+}$ is the universal operator algebra associated with the covariance relation. More precisely, it is the operator algebra generated by the image of $\CA(\alpha)$ under an isometric covariant representation $(\iota, \mathbf{v})$, subject to the following universal property: for every isometric covariant representation $(\pi, V)$, there exists a unique completely contractive homomorphism $\CA \rtimes_{\alpha} \Zb^{+} \to \Clb(\CH)$ rendering the diagram 
\begin{align*}
    \xymatrix{
        \CA(\alpha) \ar[rd]_-{\pi\rtimes_{\alpha} V} \ar[r]^{\iota \rtimes_{\alpha} \mathbf{v}} & \CA \rtimes_{\alpha} \Zb^{+} \ar[d]\\
       & \Clb(\CH)
    }
\end{align*}
commutative.

\begin{rem}\label{rem:nondeg}
Let $(\CA, \alpha)$ be a dynamical system with $\alpha$ multiplier-nondegenerate, i.e., for every $\cai$ $\{e_i\}$ of the operator algebra $\CA$, the net $\{\alpha(e_i)\}$ is again a $\cai$ of $\CA$.\footnote{Moreover, $\{e_i\}$ is a $\cai$ of $\CA \rtimes_{\alpha} \Zb^{+}$.} Suppose $(\pi, V)$ is an isometric covariant representation on a Hilbert space $\CH$. Since both $\pi(e_i)$ and $\pi(\alpha(e_i))$ converge in the $\Wot$ to the projection $P$ onto the subspace $\overline{\pi(\CA) \CH}$, we obtain
    \begin{align*}
        P V = \lim_{i} \pi(e_i) V = \lim_{i} V \pi(\alpha(e_i)) = VP. 
    \end{align*}
Hence $(P \pi( \cdot )P, VP)$ is an isometric covariant representation on $P\CH$, where $P \pi( \cdot )P$ a non-degenerate representation of $\CA$ (see, for example, Lemma 2.1.9 in \cite{BlMe04}). This shows that when $\alpha$ is multiplier-nondegenerate, we can always assume that the representation $\pi$  in an isometric covariant representation $(\pi, V)$ is non-degenerate.
\end{rem}

\begin{rem}\label{rem:diag_A}
By the universal property, the semicrossed product $\CA \rtimes_{\alpha} \Zb^{+}$ admits a gauge action
\begin{align*}
    \mathbb{T} &\to \aut(\CA \rtimes_{\alpha} \Zb^{+})\\
    \lambda &\mapsto \rho_{\lambda},
\end{align*}
where $\rho_{\lambda}$ sends $\mathbf{v}^ka$ to $\lambda^k \mathbf{v}^ka$. Then 
\begin{align*}
    \Phi_k(t) := \int_{\mathbb{T}} \rho_{\lambda}(t) \overline{\lambda^k} d \lambda, \quad t \in \CA \rtimes_{\alpha} \Zb^{+}, 
\end{align*}
defines a completely contraction from $\CA \rtimes_{\alpha} \Zb^{+}$ onto the closed subspace $\mathbf{v}^k \CA$. Moreover, the formal Fourier series $\sum_{k=0}^{\infty} \Phi_k(t)$ is Ces\'{a}ro summable to $t$, i.e.,   
\begin{align*}
    t = \lim_{n \to \infty} \sum_{k=0}^{n} \left ( 1 - \frac{k}{n+1}\right ) \Phi_k(t) 
\end{align*}
(see, for example, Chapter 2 in \cite{SS03}). 

Given a completely contractive representation $\pi: \CA \to \Clb(\CH)$, we define an isometric covariant representation  on $\CH \otimes l^2(\Zb^{+})$ defined by the complete contraction
\begin{align*}
    a \mapsto \operatorname{diag} \left (\pi(a), \pi(\alpha(a)), \pi(\alpha^2(a)), \ldots \right), \quad a \in \CA,
\end{align*}
together with the unilateral shift 
\begin{align*}
    V(\xi_0, \xi_1, \xi_2, \ldots) = (0, \xi_0, \xi_1, \xi_2, \ldots).
\end{align*}
It is then straightforward to see that if an operator $t \in \diag(\CA \rtimes_{\alpha} \Zb^{+})$ then $\Phi_k(t) = 0$ for every $k \geq 1$. Therefore, 
\begin{align*}
    \diag(\CA \rtimes_{\alpha} \Zb^{+}) = \diag (\CA).
\end{align*}
\end{rem}

The remainder of the section is devoted to show a result that implies among others that a completely isometric stable isomorphism between the semicrossed products $C_0(X) \rtimes_{\alpha} \Zb^{+}$ and $C_0(Y) \rtimes_{\beta} \Zb^{+}$ ensures an isometric isomorphism of the orginal algebras. The main ingredient of the proof is that such an isometric isomorphism induces a homeomorphism between the primitive ideal spaces of their diagonals. In what follows, we recall some basic facts about primitive ideals of C$^*$-algebras.\\

Let $\CA$ be a C$^*$-algebra and let $\FI(\CA)$ denote the collection of all closed ideals in $\CA$. An ideal of $\CA$ is called \textit{primitive} if it is the kernel of a non-zero irreducible representation of $\CA$. We denote by $\prim(\CA)$ the set of all primitive ideals of $\CA$. Endowed with the Jacobson (or hull-kernel) topology, $\prim(\CA)$ is a $T_0$-space. Furthermore the correspondence 
\begin{align*}
    \CI \mapsto \left \{\CJ\in \prim(\CA): \CI \subset \CJ \right \} 
\end{align*}
defines a bijection from $\FI(\CA)$ onto the set of closed subsets of $\prim(\CA)$ (see, for example, Section 5.4 in \cite{GM90}).

In this paper, we require an isomorphism invariant finer than $\prim(\CA)$ with its Jacobson topology. To this end, given any collection $M$ of primitive ideals of a C$^*$-algebra $\CA$, we introduce a topology on $M$ by mimicking the definition of the Jacobson topology. Let $\Cll_{M} \subseteq \FI (\CA)$ denote the meet-semilattice generated by intersections of elements of $M$. For $S \subseteq M$ and $\CJ \in \Cll_{M}$, define
\begin{align*}
    \ker S:= \bigcap_{\CI \in S} \CI \in \Cll_{M} \quad (\ker \emptyset = \CA)
\end{align*}
and 
\begin{align*}
  \hull (\CJ):=\{ \CI \in M \mid \CJ \subseteq \CI\}.  
\end{align*}
One can define a closure operation for a subset $S \subseteq M$ by 
\begin{align*}
 S':=\hull (\ker S).   
\end{align*}
By arguments identical to those of Theorem 5.4.6 in Murphy \cite{GM90}, this closure operation defines a topology on $M$. In fact, this topology coincides with the subspace topology induced on the subset $M$ by the Jacobson topology on $\prim (\CA)$ (see Remark \ref{rem:sub_top}). Therefore, we continue to refer to this topology as the  hull-kernel topology. In this paper, we consider the case $M= \mprim (\CA)$, the space of all minimal primitive ideals of $\CA$, endowed with the hull-kernel topology induced by the meet-semilattice of ideals it generates. When $\CA = C_0(X)$ is commutative, $\prim(\CA)$ and $ \mprim (\CA)$ are canonically isomorphic, as every primitive ideal of $C_0(X)$ is minimal. Beyond this case, however, these spaces may differ. 


\begin{rem}\label{rem:sub_top}
It is straightforward to verify that the hull-kernel topology on $M$ described above coincides with the subspace topology of the Jacobson topology. Indeed, for a subset $S \subset M$, 
   \begin{align*}
       \hull \ker S = \{\CJ \in \prim(\CA): \ker S \subseteq \CJ\} \cap M. 
   \end{align*}
Hence, the closure of $S$ in $M$ is precisely the intersection of $M$ with the closure of $S$ in $\prim(\CA)$. Conversely, let 
\begin{align*}
    S_0 := \{\CI \in \prim(\CA): \CJ \subseteq \CI\}
\end{align*}
be a closed subset of $\prim(\CA)$ in the Jacobson topology determined by an ideal $\CJ$ of $\CA$. Since
\begin{align*}
    \CJ \subset \ker (S_0 \cap M),  
\end{align*}
we have $(S_0 \cap M)' = S_0 \cap M$. Thus, the closed subsets of $M$ are exactly the restrictions to $M$ of closed subset of $\prim(\CA)$ on $M$.   
\end{rem}

In what follows, for a subset $S\subseteq X$ of a locally compact Hausdorff space, we write
\[
\Clz(S):= \left \{ f \in C_0(X)\mid f(s)=0, \mbox{ for all } s \in S \right\}.
\]
For $x \in X$, we suppress the braces and write $\Clz(\{x\})$ simply as $\Clz(x)$. The character of $C_0(X)$ given by evaluation at a point $x$ is denoted by $\ev_x$.

\begin{lem} \label{lem;primchar}
Let $X$ be a locally compact Hausdorff space and let $\CD$ be a primitive C$^*$-algebra. Then
\begin{itemize}
\item[(i)] Every irreducible representation of $C_0(X)\otimes \CD$ has the form $\ev_x\otimes \pi$, where $x \in X$ and $\pi$ an irreducible representation of $\CD$.

\item[(ii)] If $\ev_x\otimes \pi$ is as above and $\pi$ is faithful, then $\ker (\ev_x\otimes \pi) =\Clz(x)\otimes \CD$.

\item[(iii)] The minimal primitive ideals of $C_0(X) \otimes \CD$, i.e., the minimal elements of $\prim (C_0(X) \otimes \CD)$ under set-theoretic inclusion, are exactly the ideals 
\begin{align*}
    \Clz(x)\otimes \CD, \quad x \in X.
\end{align*}
Consequently, for a representation $\ev_x\otimes \pi$ as in (i), $\ker (\ev_x\otimes \pi)$ is minimal if and only if $\pi$ is faithful.
\end{itemize} 
\end{lem}

\begin{proof} 
    (i). Let $\sigma$ be an irreducible representation of $C_0(X)\otimes \CD$ on Hilbert space $\CH$, and let $\overline{\sigma}$ denote its extension on the multiplier algebra $\Clm(C_0(X)\otimes \CD)$. Since,
\[
\overline{\sigma}(C_0(X)\otimes I)\subseteq \sigma(C_0(X)\otimes \CD)'=\Cb I,
\]
there exists $x \in X$ so that $\overline{\sigma}(f\otimes I)=\ev_x(f)I$, for all $f \in C_0(X)$. Therefore
\[
\sigma(C_0(X)\otimes \CD)\subseteq \overline{\sigma}(C_0(X)\otimes I)\overline{\sigma}(I\otimes \CD)=\overline{\sigma}(I\otimes \CD).
\]
Thus the representation
\begin{align*}
    \pi: \CD &\longrightarrow \Clb(\CH)\\
     d &\longmapsto \overline{\sigma}(I\otimes d)   
\end{align*}
is irreducible and $\sigma = \ev_x\otimes \pi$.

(ii). It is sufficient to establish the result in the case $\pi =\id$. Consider the short exact sequence
\[
0 \longrightarrow \Clz(x)\hookrightarrow C_0(X) \xlongrightarrow{\ev_x} \Cb \longrightarrow 0.
\]
 Since $C_0(X)$ is nuclear, and hence exact, the short exact sequence 
\[
0 \longrightarrow \Clz(x)\otimes \CD\hookrightarrow C_0(X) \otimes \CD \xlongrightarrow{\ev_x\otimes \id } \Cb\otimes \CD\simeq \CD \longrightarrow 0,
\]
is also exact. Consequently, $\Clz(x)\otimes \CD$ is the kernel of the irreducible representation $\ev_x \otimes \id$.

(iii). Let 
\[
\sigma : C_0(X) \otimes \CD \longrightarrow \Clb(\CH)
\]
be an irreducible representation. Then $\sigma = \ev_x\otimes \pi$ as in (i), and hence 
\begin{align*}
     \Clz(x)\otimes \CD \subset \ker \sigma, \quad C_0(X) \otimes \ker(\pi) \subset \ker \sigma .
\end{align*}
If $\ker(\pi) \neq \{0\}$, we have
\begin{align*}
    C_0(X) \otimes \ker(\pi) \nsubseteq \Clz(x)\otimes \CD.
\end{align*}
Hence $\ker \sigma$ is not minimal. And the desired conclusion follows from the fact that 
\begin{align*}
    \Clz(x)\otimes \CD \nsubseteq \Clz(y)\otimes \CD,
\end{align*}
whenever $y \neq x$.
\end{proof}

We now prove that, if $\CD$ is a primitive C$^*$-algebra, the spaces $\mprim(C_0(X)\otimes \CD)$ and $\prim(C_0(X))\simeq X$ are canonically isomorphic.

\begin{prop} \label{prop:ideal_cor}
Let $X$ be a locally compact Hausdorff space and let $\CD$ be a primitive C$^*$-algebra. Then the map
\begin{align*}
    \phi_{X, \CD}: X &\longrightarrow \mprim(C_0(X)\otimes \CD)\\   
                   x &\longmapsto \quad \Clz(x)\otimes \CD
\end{align*}
is a homeomorphism.
\end{prop}

\begin{proof}
By Lemma~\ref{lem;primchar}, $\phi_{X, \CD}$ is a well-defined bijection. In order to prove bicontinuity, identify for the moment  $C_0(X)\otimes \CD$ canonically with the $C_0(X, \CD)$, i.e., the $\CD$-valued, continuous functions on $X$ vanishing at infinity. Under this identification, an elementary partition of unity argument shows that, for any closed set $S\subseteq X$, the ideal $\Clz(S)\otimes \CD$ is identified with the elements of $C_0(X, \CD)$ vanishing on $S$. With this in mind, given a closed set $S \subseteq X$,  one can show that 
\[
\Clz(S)\otimes\CD = \bigcap_{x \in S} \left (\Clz(x)\otimes \CD \right).
\]
Given an arbitrary subset $M \subseteq X$, we compute:
\begin{align*}
\phi_{X, \CD} \left ( \left \{\Clz(x)\mid x \in M \right\}' \right)&=\phi_{X, \CD}\left(\hull\left(\ker( \left \{\Clz(x)\mid x \in M \right\}\right) \right)\\
            &=\phi_{X, \CD}\left (\{\Clz(x)\mid x \in \overline{M}\} \right) \\
            &= \left \{\Clz(x)\otimes \CD\mid x \in \overline{M} \right\}\\
            &= \hull \left(\Clz \left (\overline{M} \right)\otimes\CD \right)\\
            &=\hull\left(\bigcap_{x \in \overline{M}} \left (\Clz(x)\otimes \CD \right)\right)\\
            &= \hull\left(\bigcap_{x \in M} \left (\Clz(x)\otimes \CD \right)\right)\\
            &=\hull\left(\ker \left (\{\Clz(x) \otimes \CD\mid x \in M\} \right )\right)  \\
            &= \left \{\phi_{X, \CD}(\Clz(x))\mid x \in M \right\}'.
\end{align*}
Hence, $\phi_{X, \CD}$ preserves closures, and its bicontinuity follows.
\end{proof}


A $*$-homomorphism $\gamma: \CA \to \CB$ between C$^*$-alegbras is said to \textit{preserve irreducible representations} if, for every nonzero irreducible representation $\varphi$ of $\CB$, the composition $\varphi \circ \gamma$ is a nonzero irreducible representation of $\CA$. In this case, $\gamma$ induces a map 
\begin{align*}
    \widehat{\gamma} : \quad \prim(\CB) &\longrightarrow \prim (\CA) \\
                        \ker(\varphi) \quad  &\longmapsto  \quad \ker(\varphi \circ \gamma)= \gamma^{-1}(\ker(\varphi)).
\end{align*}
It is clear that every surjective $*$-homomorphism $\gamma: \CA \to \CB$ preserves irreducible representations. In particular, if $\gamma$ is a $*$-isomorphism, the induced map $\hat{\gamma}: \prim(\CB) \to \prim(\CA)$ is a homeomorphism. However, for a general map that merely preserves irreducible representations, continuity of the map $\hat{\gamma}$ is not automatic. Nevertheless, in our case, for a C$^*$-dynamical system $(\CA, \alpha)$, by restricting to a carefully chosen subset of $\prim (\CA)$, we can ensure that $\widehat{\alpha}$ is continuous on that subset.\newline

Before proceeding, we pause to present a result that explains why our analysis is restricted to non-degenerate $*$-endomorphisms of C$^*$-algebras.

\begin{prop}\label{prop:pr_nd}
   Let $\gamma: \CA \to \CB$ be a $*$-homomorphism between C$^*$-algebras. If $\gamma$ preserves irreducible representations, then $\gamma$ is non-degenerate. Conversely, if $\CA$ and $\CB$ are commutative C$^*$-algebras and $\gamma$ is non-degenerate, then  $\gamma$ preserves irreducible representations.
\end{prop}

\begin{proof}
Assume that $\gamma$ is degenerate, then $\overline{\CB \gamma(\CA)}$ is a proper closed left ideal in $\CB$. There exists a pure state $\rho$ of $\CB$ such that 
    \begin{align*}
        \rho(t^*t) = 0, \quad \forall t \in \overline{\CB\gamma(\CA)} 
    \end{align*}
(see, e.g., Theorem 5.3.3 in \cite{GM90}). Let $\pi: \CB \to \CH_{\rho}$ be the GNS representation associated with the pure state $\rho$, and let $\xi_{\rho} \in \CH_{\rho}$ be the unique cyclic unit vector implementing $\rho$, i.e., 
    \begin{align*}
        \rho(b) = \braket{\xi_{\rho}}{\pi(b) \xi_\rho}, \quad \forall b \in \CB
    \end{align*}
(see, e.g., Theorem 5.1.1 in \cite{GM90}). Since $\rho(\gamma(aa^*)) = 0$ for every $a \in \CA$, it follows that $\pi \circ \gamma(\CA) \xi_{\rho} = \{0\}$. Thus $\pi \circ \gamma$ cannot be an irreducble representation of $\CA$. Recall that $\pi$ is an irreducible representation of $\CB$. We conclude that $\gamma$ does not preserve irreducible representations. 

Conversely, let $X, Y$ be locally compact Hausdorff spaces so that $\CA=C_0(X)$ and $\CB=C_0(Y)$. Let $\gamma: C_0(X) \to C_0(Y)$ be a $*$-homomorphism. Since $\gamma$ is non-degenerate, there exists a proper continuous map $\widehat{\gamma}: Y \to X$ such that
    \begin{align*}
        \gamma(f) = f \circ \widehat{\gamma}, \quad \forall f \in C_0(X)
    \end{align*}
(see, for example, Corollary C.48 in \cite{KL17}). Then 
\begin{align*}
    \ev_y \circ \gamma = \ev_{\widehat{\gamma}(y)},
\end{align*}
which is irreducible. 
\end{proof}

\begin{rem}
The third sentence of Proposition~\ref{prop:pr_nd} almost never holds in general, as illustrated by the simple example
\begin{align*}
    \Cb &\to M_2(\Cb)\\
    \lambda &\mapsto \lambda I_2.
\end{align*}   
\end{rem}

\begin{prop}
Let $X$ be a locally compact Hausdorff space, and let $\CD$ be a primitive C$^*$-algebra. Suppose $\alpha$ is a non-degenerate $*$-endomorphism of $C_0(X)$ induced by a proper continuous self map $\widehat{\alpha}$ of $X$, i.e., $\alpha(f) = f \circ \widehat{\alpha}$. Then 
\begin{equation}  \label{eq;incl}
\widehat{\alpha\otimes \id}\left(\mprim(C_0(X)\otimes \CD)\right)\subseteq \mprim(C_0(X)\otimes \CD)
\end{equation}
and the restriction of $\widehat{\alpha\otimes \id}$ on $\mprim(C_0(X)\otimes \CD)$, denoted again as $\widehat{\alpha\otimes \id}$ is continuous with respect to the hull-kernel topology.
\end{prop}

\begin{proof}
By viewing $C_0(X)\otimes \CD$ as continuous, $\CD$-valued functions vanishing at infinity, one can easily see that 
\begin{equation*}
(\alpha\otimes \id)^{-1}(\Clz(x)\otimes \CD)=\Clz(\widehat{\alpha}(x))\otimes \CD, \quad x \in X.
\end{equation*}
Combined with Lemma~\ref{lem;primchar}, this proves \eqref{eq;incl}. Finally if $\phi_{X, \CD}$ is the homeomorphism of Proposition~\ref{prop:ideal_cor}, then $\widehat{\alpha\otimes \id}= \phi_{X, \CD} \circ \widehat{\alpha} \circ \phi_{X, \CD}^{-1}$. Therefore $\widehat{\alpha\otimes \id}$ is continuous.
\end{proof}

Let $\CA$ be an operator algebra. We denote by $\rep(\CA)$ the set of non-degenerate, contractive representations of $\CA$ of the form
\begin{align*}
    \CA &\to \Clb(\CH \oplus \CH) \\
    a &\mapsto 
   \begin{pmatrix}
     \varphi_{1}(a) & 0\\
      \varphi_3(a) & \varphi_2(a)
   \end{pmatrix}, 
\end{align*}
where the following hold:
\begin{itemize}
    \item $\varphi_1$ and $\varphi_2$ restrict to irreducible representations of $\diag(\CA)$ on $\CH$;
    \item $\varphi_3$ is nonzero.
\end{itemize}

\begin{rem}
   Let 
   \begin{align*}
   a \mapsto 
   \begin{pmatrix}
     \varphi_{1}(a) & 0 \\
     \varphi_3(a) & \varphi_2(a)
   \end{pmatrix}
\end{align*} 
be a contractive representation in $\rep(\CA)$. For every $a\in \diag(\CA)$, we have $\varphi_3(a) = 0$ since the contractive homomorphism maps the diagonal $\diag(\CA)$ into the diagonal of the operator algebra
        \begin{align*}
            \left \{ 
            \begin{pmatrix}
                A_{11} & 0\\
                A_{21} & A_{22}
            \end{pmatrix} : A_{11}, A_{21}, A_{22} \in \Clb(\CH) \right \}.
        \end{align*} 
\end{rem}

\begin{example}\label{exam:rep_exp}
    Let $(\CA, \alpha)$ be a C$^*$-dynamical system such that $\alpha$ preserves irreducible representations of $\CA$. For each irreducible representation $\pi: \CA \to \Clb(\CH)$, consider the restriction of the Fock representation of $\CA \rtimes_{\alpha} \Zb^{+}$ on $\CH \otimes l^2(\Zb^{+})$ (see Remark~\ref{rem:diag_A}) to the upper-left $2\times 2$ concer. This gives a representation $\Phi_{\pi} \in \rep(\CA \rtimes_{\alpha} \Zb^{+})$ defined by 
    \begin{align*}
    \Phi_{\pi}(a_0 + \mathbf{v}a_1 + \mathbf{v}^2 a_2 + \cdots) :=   
       \begin{pmatrix}
           \pi(a_0) & 0\\ 
           \pi(a_1) & \pi(\alpha(a_0))
       \end{pmatrix}, \quad a_i \in \CA.
\end{align*}\newline
\end{example}

In what follows, we say that an operator algebra $\CB$ has an \textit{irreducibly acting} diagonal if $\CB$ admits a contractive representation on a Hilbert space whose restriction to the diagonal is both faithful and irreducible. In particular, $\diag(\CB)$ is primitive. Examples of such algebras include all (represented) operator algebras containing the compact operators, which were considered in the previous section.


\begin{lem}\label{lem:diag_rel} 
Let $(C_0(X), \alpha)$ be a C$^*$-dynamical system in which $\alpha$ is induced by a proper continuous map $\widehat{\alpha}$, and let $\CB$ be an approximately unital operator algebra with an irreducibly acting diagonal. Suppose  
    \begin{align*}
        \Phi(t):= \begin{pmatrix}
              \varphi_1(t) & 0\\
              \varphi_3(t) & \varphi_2(t)
        \end{pmatrix} \in \Clb(\CH\otimes \CH), 
        \quad t \in (C_0(X) \rtimes_{\alpha} \Zb^{+})\otimes \CB,
    \end{align*}
is a representation in $\rep((C_0(X) \rtimes_{\alpha} \Zb^{+})\otimes \CB)$ with 
\[\varphi_i |_{C_0(X)\otimes \diag (B)}=\ev_{x_i}\otimes \pi_i, \]
where $x_i \in X$ and $\pi_i:\diag (B) \rightarrow \Clb(\CH)$ are faithful irreducible representations for $i=1, 2$. If $x_1\neq x_2$, then $\widehat{\alpha}(x_1)=x_2$.
\end{lem}

\begin{proof}
Note that $(C_0(X) \rtimes_{\alpha} \Zb^{+})\otimes \CB$ demonstrates a semicrossed product-like structure. Indeed the isometry $\mathbf{w}:=\mathbf{v} \otimes I \in \Clm ((C_0(X) \rtimes_{\alpha} \Zb^{+})\otimes \CB)$ implements the action of $\alpha\otimes \id$ on $C_0(X)\otimes \CB \subseteq (C_0(X) \rtimes_{\alpha} \Zb^{+})\otimes \CB$, i.e., 
\[
    (f\otimes b) \mathbf{w}=\mathbf{w}(\alpha(f)\otimes b) = \mathbf{w}(f\circ\widehat{\alpha} \otimes b), \quad f \in C_0(X), b \in \CB.
\]
Furthermore, $(C_0(X) \rtimes_{\alpha} \Zb^{+})\otimes \CB$ is generated by $C_0(X) \otimes \CB$ and $\mathbf{w}$ in the sense that polynomials of the form
\[
f_0\otimes b_0+\mathbf{w}(f_1\otimes b_1)+\mathbf{w}^2 (f_2\otimes b_2)+\dots + \mathbf{w}^n(f_n\otimes b_n),
\]
$f_i \in C_0(X)$, $b_i \in \CB$, $0 \leq i \leq n$, are dense in $(C_0(X) \rtimes_{\alpha} \Zb^{+})\otimes \CB$.

Note that $(C_0(X) \rtimes_{\alpha} \Zb^{+})\otimes \CB$ is approximately unital. Let 
\[
\overline{\Phi}:\Clm((C_0(X) \rtimes_{\alpha} \Zb^{+})\otimes \CB)\longrightarrow \Clb(\CH\otimes \CH)
\]
    be the multiplier extension of $\Phi$ (see Proposition~\ref{prop:ext_to_double_dual}). We now establish the following two claims about $\overline{\Phi}$ before proceeding.

\vspace{.1in}
\noindent \textbf{Claim 1:}
$\overline{\Phi}(f \otimes I)=
        (\begin{smallmatrix}
          f(x_1)I & 0\\
         0 & f(x_2)I
        \end{smallmatrix})$, for all $f \in C_0(X)$.
\vspace{.1in}

\noindent By Lemma 2.9 in \cite{KaKaLi23}, 
\begin{align*}
    \diag((C_0(X) \rtimes_{\alpha} \Zb^{+})\otimes \CB) = C_0(X) \otimes \diag (\CB).
\end{align*}
Since $\ev_{x_i} \otimes \pi_i$, $i =1 ,2$ are irreducible, the restriction $\Phi|_{C_0(X) \otimes \diag (\CB)}$ is non-degenerate, i.e., 
\[
\Span \left\{(\Phi(g\otimes d)\xi\mid g \in C_0(X), d \in \diag(\CB), \xi \in \CH\otimes \CH \right\}
\]
is dense in $\CH\otimes \CH$. Then the claim follows from the following computation:
\begin{align*}
\overline{\Phi}(f \otimes I)(\Phi(g)\otimes d)\xi&=\Phi(fg\otimes d)\xi \\
                        &=\begin{pmatrix}
                        \varphi_1(fg\otimes d)& 0\\
                         0 & \varphi_2(fg\otimes d)
                         \end{pmatrix}\xi \\           
                        &=\begin{pmatrix}
                        fg(x_1)\otimes \pi_1(d)& 0\\
                         0 & fg(x_2)\otimes \pi_2(d)
                         \end{pmatrix}\xi  \\
                         &=\begin{pmatrix}
                        f(x_1)I& 0\\
                         0 & f(x_2)I
                         \end{pmatrix} \begin{pmatrix}
                        g(x_1)\otimes \pi_1(d)& 0\\
                         0 & g(x_2)\otimes \pi_2(d)
                         \end{pmatrix}\xi \\
                         &= \begin{pmatrix}
                        f(x_1)I& 0\\
                         0 & f(x_2) I
                         \end{pmatrix} \Phi(g\otimes d)\xi.
\end{align*}

\vspace{.1in}
\noindent \textbf{Claim 2:}
$\overline{\Phi}(\mathbf{w})=
       ( \begin{smallmatrix}
          X& 0\\
         Z& Y
        \end{smallmatrix})$ with $Z \neq0$.
       \vspace{.1in}

\noindent Indeed, by Claim 1, $\overline{\Phi}(C_0(X)\otimes I)=\Cb I\oplus \Cb I$, because $x_1\neq x_2$. Since 
\[
\Phi(C_0(X)\otimes \CB)\subseteq \overline{\Phi}(C_0(X)\otimes I)'
\]
we conclude that the (2,1)-entry in the block matrix presentation of $\Phi(C_0(X)\otimes \CB)$ is $0$. If $Z=0$, then $\overline{\Phi}(\mathbf{w})$ is also diagonal, and consequently $\Phi \left((C_0(X) \rtimes_{\alpha} \Zb^{+})\otimes \CB)\right) $ consists entirely of diagonal operators. This contradicts the assumption that $\varphi_3$ is nonzero, and hence the claim is proved. 

\vspace{.1in}

By comparing the $(2,1)$-entry of the matrices in the equation 
\begin{align*}
    \begin{pmatrix}
    \varphi_1(s) & 0\\
    0 & \varphi_2(s) 
\end{pmatrix}    
\begin{pmatrix}
    X & 0\\
    Z & Y
\end{pmatrix} 
&= \Phi(s)\overline{\Phi}(\mathbf{w})= \Phi(s\mathbf{w})\\
&=  \Phi(\mathbf{w} \left(\alpha\otimes \id(s))\right)\\
&= \overline{\Phi}(\mathbf{w})\Phi(\alpha\otimes \id(s))   \\
&= \begin{pmatrix}
    X & 0\\
    Z & Y
\end{pmatrix}
    \begin{pmatrix}
        \varphi_1(\alpha \otimes \id(s)) & 0\\
        0 & \varphi_2(\alpha \otimes \id(s)) 
\end{pmatrix},
\end{align*}
we obtain 
\begin{align*}
 \varphi_2(s)Z = Z\varphi_1(\alpha\otimes \id (s)), \, \mbox{ for all } s \in C_0(X)\otimes \diag(\CB).   
\end{align*}
Since $Z \neq 0$, Schur’s Lemma implies that the two irreducible representations 
\[
\varphi_1 \circ (\alpha\otimes \id) |_{C_0(X)\otimes \diag(\CB)}=\ev_{\widehat{\alpha}(x_1)}\otimes \pi_1 \mbox{ and } \varphi_2|_{C_0(X)\otimes \diag(\CB)}= \ev_{x_2}\otimes \pi_2
\] are unitarily equivalent and thus they have the same kernels. Recall that $\diag(\CB)$ is primitive. By Lemma~\ref{lem;primchar} (ii)
\[
\Clz \left(\widehat{\alpha}(x_1) \right)\otimes \diag(\CB)= \Clz(x_2)\otimes \diag(\CB)
\]
and hence $\widehat{\alpha}(x_1)=x_2$, as desired.
\end{proof}

Recall that two C$^*$-dynamical system $(\CA, \alpha)$ and $(\CB, \beta)$ are \textit{conjugate} if there exists a $*$-isomorphism $\gamma :\CA \to \CB$ such that 
\begin{align*}
 \alpha = \gamma^{-1} \circ \beta \circ \gamma.   
\end{align*}
Let $(\iota_{\CA}, \mathbf{v}_{\CA})$ (resp. $(\iota_{\CB}, \mathbf{v}_{\CB})$) denote the isometric covariant representation of $(\CA, \alpha)$ (resp. $(\CB, \beta)$) such that $\iota_{\CA} \rtimes \mathbf{v}_{\CA} (\CA(\alpha))$ (resp. $\iota_{\CB} \rtimes \mathbf{v}_{\CB} (\CB(\beta))$) generates $\CA \rtimes_{\alpha} \Zb^{+}$ (resp. $\CB \rtimes_{\beta} \Zb^{+}$). By the universal property of the semicrossed product, the isometric covariant representation $(\iota_{\CB}\circ \gamma, \mathbf{v}_{\CB})$ of $(\CA, \alpha)$ induces a completely isometric isomorphism 
\begin{align*}
    \CA \rtimes_{\alpha} \Zb^{+} \xrightarrow{\sim} \CB \rtimes_{\beta} \Zb^{+},
\end{align*}
whose inverse is induced by the isometric covariant representation $(\iota_{\CA} \circ \gamma^{-1}, \mathbf{v}_{\CA})$ of $(\CB, \beta)$.

\begin{thm} \label{mainthm;semicr}
Let $(C_0(X), \alpha)$ and $(C_0(Y), \beta)$ be two C$^*$-dynamical systems with $\alpha$ and $\beta$ non-degenerate, and let $\CB_X$ and $\CB_Y$ be operator algebras with irreducibly acting diagonals. If the operator algebras 
  \[
  (C_0(X) \rtimes_{\alpha} \Zb^{+} ) \otimes \CB_X \,\,\, \mbox{ and } \,\,\,(C_0(Y) \rtimes_{\beta} \Zb^{+} ) \otimes \CB_Y
  \]
  are isometrically isomorphic, then the C$^*$-dynamical systems 
  \[
  (C_0(X), \alpha)  \,\,\, \mbox{ and } \,\,\, (C_0(Y), \beta)
  \]
  are conjugate
 \end{thm}
 
 \begin{proof}
By Lemma 2.9 in \cite{KaKaLi23}, 
\begin{align*}
    &\diag((C_0(X) \rtimes_{\alpha} \Zb^{+})\otimes \CB_X) = C_0(X) \otimes \diag (\CB_X), \\
    &\diag((C_0(Y) \rtimes_{\alpha} \Zb^{+})\otimes \CB_Y) = C_0(Y) \otimes \diag (\CB_Y).
\end{align*}
 Let 
 \[
 \gamma: \left( C_0(X) \rtimes_{\alpha} \Zb^{+} \right) \otimes \CB_X \longrightarrow (C_0(Y) \rtimes_{\beta} \Zb^{+} ) \otimes \CB_Y
 \]
 be an isometric isomorphism. Then $\gamma$ restricts to a $*$-isomorphism from $C_0(X) \otimes \diag (\CB_X)$ onto $C_0(Y) \otimes \diag (\CB_Y)$, and induces a homeomorphism
 \begin{align*}
     \widetilde{\gamma} : \mprim(C_0(Y)\otimes \diag(\CB_Y))  &\longrightarrow \mprim(C_0(X)\otimes\diag(\CB_X))\\
     \CJ \qquad &\longmapsto \quad \gamma^{-1}(\CJ).
 \end{align*}
Recall that both $\diag(\CB_X)$ and $\diag(\CB_Y)$ are primitive. By Proposition~\ref{prop:ideal_cor}, there exists a homeomorphism $\widehat{\gamma}: Y\rightarrow X$ such that
\begin{align*}
 \widetilde{\gamma}\left(\Clz(y)\otimes \diag(\CB_Y)\right) = \Clz(\,\widehat{\gamma}(y))\otimes \diag(\CB_X), \quad \forall y \in Y.
\end{align*}

Let $\widehat{\alpha}$ and $\widehat{\beta}$ be the proper continuous maps implementing $\alpha$ and $\beta$, respectively. We will show that $\widehat{\gamma}$ implements the desired conjugacy between the systems $(C_0(X), \alpha)$ and $(C_0(Y), \beta)$ by verifying that
\begin{equation} \label{eq;desired_conj}
\widehat{\gamma}\circ \widehat{\beta}=\widehat{\alpha}\circ \widehat{\gamma}.
\end{equation}

Choose and fix a contractive representation $\pi$ of $\CB_Y$ such that its restriction to $\diag(\CB_Y)$ is a faithful irreducible representation. For each $y \in Y$ with $\widehat{\beta}(y)\neq y$, let $\Phi_{\ev_y}$ denote the presentation of $C_0(Y) \otimes_{\beta} \Zb^{+}$ defined in Example~\ref{exam:rep_exp}, where $\ev_y$ is the evaluation at $y$. Then 
\[
\Phi_{\ev_y}\otimes \pi \in \rep\left( (C_0(Y) \rtimes_{\beta} \Zb^{+} ) \otimes \CB_Y\right).
\]
Note that, for every $z \in Y$, 
\[
    \ker\left((\ev_z\otimes \pi)\circ\gamma\right) = \widetilde{\gamma} ( \ker (\ev_z\otimes \pi) )=\Clz \left (\widehat{\gamma}(z)\right)\otimes \diag(\CB_X).
\]
Then, by Lemma~\ref{lem;primchar}, there exists a faithful irreducible representation $\pi_{\gamma, z}$ such that 
\begin{equation} \label{eq;2matr}
(\ev_z\otimes \pi)\circ\gamma = \ev_{\widehat{\gamma}(z)}\otimes \pi_{\gamma, z}.
\end{equation}
Since $(\Phi_{\ev_y}\otimes \pi)\circ \gamma \in \rep \left(C_0(X) \rtimes_{\alpha} \Zb^{+} ) \otimes \CB_X\right)$, we have, for every $s \in C_0(X)\otimes \diag(B_X)$,
\begin{equation} \label{2m}
\begin{aligned}
\left( (\Phi_{\ev_y}\otimes \pi)\circ \gamma\right) (s)&= 
 \begin{pmatrix}
       \left( ( \ev_y\otimes \pi) \circ \gamma\right)(s) & 0\\
        0 &  \left( ( \ev_{\widehat{\beta}(y)}\otimes \pi) \circ \gamma\right)(s) 
        \end{pmatrix} \\
        &=\begin{pmatrix}
        \left ( \ev_{\widehat{\gamma}(y)}\otimes \pi_{\gamma, y} \right)(s) & 0\\
        0 &  \left( \ev_{\widehat{\gamma}(\widehat{\beta}(y))}\otimes \pi_{\gamma, \widehat{\beta}(y)}\right)(s) 
        \end{pmatrix}, 
\end{aligned}
\end{equation}
with the second equality following from \eqref{eq;2matr}. Recall that $\widehat{\gamma}$ is a homeomorphism. Hence $\widehat{\gamma}(y)\neq\widehat{\gamma}(\widehat{\beta}(y))$. By Lemma~\ref{lem:diag_rel}, it follows from \eqref{2m} that 
\begin{equation*} 
\widehat{\alpha}(\widehat{\gamma}(y)) = \widehat{\gamma} (\widehat{\beta}(y))
\end{equation*}
for every $y$ with $\widehat{\beta}(y) \neq y$. In particular, $\widehat{\gamma}(y) \neq \widehat{\alpha}(\widehat{\gamma}(y))$, so $\widehat{\gamma}$ maps non-fixed points to non-fixed points. By symmetry, the same holds for $\widehat{\gamma^{-1}}=\widehat{\gamma}^{-1}$. Hence $\widehat{\gamma}$ must map fixed points to fixed points. Therefore, the \eqref{eq;desired_conj} holds.
\end{proof}

\begin{cor}\label{thm:dy_case}
    Let $(C_0(X), \alpha)$ and $(C_0(Y), \beta)$ be two C$^*$-dynamical systems with $\alpha$ and $\beta$ non-degenerate, and let $\CC_{X}$ and $\CC_{Y}$ be primitive C$^*$-algebras. Then the following statements are equivalent: 
    \begin{enumerate}
        \item[(i)] $\left (C_0(X) \rtimes_{\alpha} \Zb^{+} \right) \otimes \CC_{X}$ and $\left (C_0(Y) \rtimes_{\beta} \Zb^{+} \right) \otimes \CC_{Y}$ are (completely) isometrically isomorphic;
        \item[(ii)] the C$^*$-dynamical systems $(C_0(X), \alpha)$ and $(C_0(Y), \beta)$ are conjugate and the C$^*$-algebras $\CC_{X}$ and $\CC_{Y}$ are $*$-isomorphic.
    \end{enumerate}
\end{cor}

\begin{proof}
In light of Theorem \ref{mainthm;semicr}, it suffices to show that the existence of an isometric isomorphism
\begin{align*}
    \gamma: \left (C_0(X) \rtimes_{\alpha} \Zb^{+} \right) \otimes \CC_{X} \to \left (C_0(Y) \rtimes_{\beta} \Zb^{+} \right) \otimes \CC_{Y}
\end{align*}
ensures that $\CC_X$ and $\CC_Y$ are $*$-isomorphic. Without loss of generality, we may assume that $\CC_Y$ acts irreducibly on a Hilbert space $\CH$. By the proof of Theorem~\ref{mainthm;semicr}, there exists $x \in X$ and a faithful $*$-representation $\pi: \CC_X \to \Clb(\CH)$ such that
    \begin{align*}
        (\ev_x \otimes \pi)(s) = (\ev_y \otimes \id) \circ \gamma(s), \quad \forall s \in C_0(X) \otimes \CC_X. 
    \end{align*}
Therefore,
\begin{align*}
    \CC_X \simeq  (C_0(X) \otimes \CC_X)/ (\Clz(x) \otimes \CC_X) \simeq  (\ev_y \otimes \id)\left (
    C_0(Y) \otimes \CC_Y \right) \simeq \CC_Y.
\end{align*}
\end{proof}

\begin{question}
    The Corollary~\ref{thm:dy_case} strengthens a recent result of Kakariadis, Katsoulis and Li (Corollary 3.5 in \cite{KaKaLi23}), and can be further examined from at least two perspectives. The first is to consider whether the statements in the theorem still hold if the condition that the $*$-endomorphisms are non-degenerate is removed. The second is to consider the case of multivariable C$^*$-dynamical systems. A multivariable C$^*$-dynamical system is a pair $(\CA, (\alpha_1, \ldots, \alpha_n))$, where $\CA$ is a C$^*$-algebra and each $\alpha_i$ is a $*$-endomorphism of $\CA$. The notions of tensor algebras and semicrossed products associated with multivariable C$^*$-dynamical systems were introduced in \cite{DK11} as generalizations of the semicrossed products of single-variable C$^*$-dynamical systems. In light of the previous results in this section, it is natural to ask whether stably isomorphic tensor algebras of multivariable commutative dynamical systems are necessarily isomorphic.
\end{question}

In the proof of Lemma~\ref{lem:diag_rel}, we commented that algebras of the form $(C_0(X) \rtimes_{\alpha} \Zb^{+})\otimes \CB$ demonstrate a semicrossed product-like structure. Our next result makes this comment very precise for the most significant case addressed in this paper.

\begin{prop}\label{prop:semi_K_iso}
    Let $(\CA, \alpha)$ be a C$^*$-dynamical system with $\alpha$ non-degenerate. Then 
    \begin{align*}
        (\CA \rtimes_{\alpha} \Zb^{+}) \otimes \Clk \simeq (\CA \otimes \Clk) \rtimes_{\alpha \otimes \id} \Zb^{+}. 
    \end{align*}
\end{prop}

\begin{proof}
Assume that $(\iota, \mathbf{v})$ is the universal isometric covariant representation associated with the C$^*$-dynamical system $(\CA, \alpha)$. Then $(\iota \otimes \id, \mathbf{v} \otimes I)$ is an isometric covariant representation of the C$^*$-dynamical system $(\CA \otimes \Clk, \alpha \otimes \id)$, and the operator algebra generated by the image of $(\iota \otimes \id) \rtimes_{\alpha \otimes \id} (\mathbf{v} \otimes I)$ is $(\CA \rtimes_{\alpha} \Zb^{+}) \otimes \Clk$. We prove the proposition by showing that $(\CA \rtimes_{\alpha} \Zb^{+}) \otimes \Clk$ satisfies the same universal property as $(\CA \otimes \Clk) \rtimes_{\alpha \otimes \id} \Zb^{+}$. 

Given an isometric covariant representation $(\pi, V)$ of $(\CA \otimes \Clk, \alpha \otimes \id)$ on $\Clb(\CH)$ such that $\pi$ is non-degenerate (see Remark~\ref{rem:nondeg}), the map $\pi$ extends uniquely to a $*$-homomorphism 
\begin{align*}
    \overline{\pi}: \Clm(\CA) \otimes \Clb(\CK) \to \Clb(\CH),
\end{align*}
since $\Clm(\CA) \otimes \Clb(\CK) \subset \Clm(\CA \otimes \Clk)$ (see Proposition~\ref{prop:ext_to_mul} or Proposition 2.1 in \cite{EL95}). Recall that the non-degenerate representation $\overline{\pi}(I \otimes -): \Clk \to \Clb(\CH)$ is equivalent to a multiple of the identity representation of $\Clk$, i.e., there exists a Hilbert space $\CH_0$ and a unitary $U: \CH_0 \otimes \CK \to \CH$ such that 
\begin{align*}
    U^*\overline{\pi}(I \otimes K)U = I \otimes K, \quad \forall K \in \Clk
\end{align*}
(see, for example, \cite[Corollary 1 in Section 1.4]{WA76}). By replacing $\pi$ with $U^* \pi(\cdot)U$ and $V$ with $U^*VU$, we may assume that $\CH = \CH_0 \otimes \CK$ and 
\begin{align*}
    \overline{\pi}(I \otimes K) = I \otimes K, \quad \forall K \in \Clk.
\end{align*}

    Let $\{E_{kl}\}_{k,l \in \Nb}$ be a system of matrix units for $\Clb(\CK)$, and let $\{E_i\}_{i \in \mathbb{I}}$ be an approximate unit of $\CA$. Then $\{E_i \otimes F_j\}_{(i,j) \in \mathbb{I} \times \Nb}$ forms an approximate unit of $\CA \otimes \Clk$, where $F_j = \sum_{k=1}^j E_{kk}$. Note that, for every $K \in \Clk$, we have 
\begin{align*}
    (I \otimes K) V &= \lim_{(i,j)} \pi(E_i \otimes F_j) \overline{\pi}(I \otimes K) V\\
    &=\lim_{(i,j)} \pi(E_i \otimes F_jK) V =  \lim_{(i,j)} V \pi(\alpha(E_i) \otimes F_jK)\\
    &= V\lim_{(i,j)} \pi(\alpha(E_i) \otimes F_j) \overline{\pi}(I \otimes K) = V (I \otimes K),
\end{align*}
since
\begin{align*}
    \pi(E_i \otimes F_j) \xrightarrow{\Wot} I, \quad \pi(\alpha(E_i) \otimes F_j) \xrightarrow{\Wot} I.
\end{align*}
Therefore, $V = V_0 \otimes I$, where $V_0$ is an isometry in $\Clb(\CH_0)$. Similarly, for all $A \in \CA$ and $K \in \Clk$, 
\begin{align*}
    \overline{\pi}(A \otimes I) (I \otimes K) = \pi(A \otimes K) = (I \otimes K) \overline{\pi}(A \otimes I), 
\end{align*}
so there exists a $*$-homomorphism $\pi_0: \CA \to \Clb(\CH_0)$ such that 
\begin{align*}
    \overline{\pi}(A \otimes I) = \pi_0(A) \otimes I, \quad \forall A \in \CA. 
\end{align*}
By
\begin{align*}
    (\pi_0(A)V_0) \otimes E_{11} = \pi(A \otimes E_{11}) V = V \pi(\alpha(A) \otimes E_{11}) = (V_0 \pi_0(\alpha(A)) \otimes E_{11},  
\end{align*}
    we conclude that $\pi_0(A) V_0 = V_0 \pi(\alpha(A))$, so $(\pi_0, V_0)$ is an isometric covariant representation of $(\CA, \alpha)$. Let $\varphi: \CA \rtimes_{\alpha} \Zb^{+} \to \Clb(\CH_0)$ be the completely contractive homomorphism such that 
    \begin{align*}
        \pi_0 \rtimes_{\alpha} V_0 = \varphi \circ (\iota \rtimes_{\alpha} \mathbf{v}). 
    \end{align*}
Then $\varphi \otimes \id: (\CA \rtimes_{\alpha} \Zb^{+}) \otimes \Clk \to \Clb(\CH_0 \otimes \CK)$ is a completely contractive homomorphism satisfying 
\begin{align*}
    \pi \rtimes_{\alpha} V = (\varphi \otimes \id) \circ [(\iota \otimes \id) \rtimes_{\alpha} (\mathbf{v} \otimes I)]. 
\end{align*}
Hence, $(\CA \rtimes_{\alpha} \Zb^{+}) \otimes \Clk$ is the universal operator algebra associated with the C$^*$-dynamical system $(\CA \otimes \Clk, \alpha \otimes \id)$.  
\end{proof}

\noindent\textbf{Acknowledgements.} 
Research by the first author (E. G. Katsoulis) is supported by the NSF Award 2054781.
Research by the second author (F. Miao) is supported by the NSFC under grant number 12426650, 12426664, 12571141.
Research by the third author (W. Wu) is supported by the NSFC under grant numbers 12571130.
Research by the fourth author (W. Yuan) is supported by the NSFC under grant numbers 12471124, 12571130.

\bibliographystyle{amsplain}
\bibliography{Bib}

@Article{Arv1,
  author  = {Arveson, W.},
  title   = {Operator algebras and measure preserving automorphisms},
  journal = {Acta Math.},
  volume  = {118},
  year    = {1967},
  pages   = {95--109},
  doi     = {10.1007/BF02392478},
}

@article{Arv2,
  author  = {Arveson, W.},
  title   = {Operator algebras and invariant subspaces},
  journal = {Ann. of Math. (2)},
  volume  = {100},
  year    = {1974},
  pages   = {433--532},
  doi     = {10.2307/1970956},
}

@Book{DavNest,
  author    = {Davidson, K.},
  title     = {Nest Algebras},
  series    = {Pitman Research Notes in Mathematics Series},
  volume    = {191},
  note      = {Triangular forms for operator algebras on Hilbert space},
  publisher = {Longman Scientific \& Technical},
  address   = {Harlow},
  year      = {1988},
  pages     = {xii+412},
  isbn      = {0-582-01993-1},
}

@Article{DK1,
  author  = {Davidson, K. and Katsoulis, E.},
  title   = {Semicrossed products of simple {$C^*$}-algebras},
  journal = {Math. Ann.},
  volume  = {342},
  number  = {3},
  year    = {2008},
  pages   = {515--525},
  doi     = {10.1007/s00208-008-0244-1},
}

@book{SZ19,
  author    = {Str\u{a}til\u{a}, S. and Zsid\'o, L.},
  title     = {Lectures on von {N}eumann Algebras},
  edition   = {2},
  series    = {Cambridge IISc Series},
  publisher = {Cambridge University Press},
  year      = {2019},
  doi       = {10.1017/9781108654975},
}

@Article{DKak,
  author  = {Davidson, K. and Kakariadis, E. T. A.},
  title   = {Conjugate dynamical systems on {$C^*$}-algebras},
  journal = {Int. Math. Res. Not.},
  number  = {5},
  year    = {2014},
  pages   = {1289--1311},
  doi     = {10.1093/imrn/rns253},
}

@Article{KRSigma,
  author  = {Katsoulis, E. and Ramsey, C.},
  title   = {The isomorphism problem for tensor algebras of multivariable dynamical systems},
  journal = {Forum Math. Sigma},
  volume  = {10},
  year    = {2022},
  pages   = {e81},
  doi     = {10.1017/fms.2022.73},
}

@Article{DK3,
  author  = {Davidson, K. and Katsoulis, E.},
  title   = {Dilation theory, commutant lifting and semicrossed products},
  journal = {Doc. Math.},
  volume  = {16},
  year    = {2011},
  pages   = {781--868},
}

@Article{Pow1,
  author  = {Power, S. C.},
  title   = {Classification of analytic crossed product algebras},
  journal = {Bull. Lond. Math. Soc.},
  volume  = {24},
  number  = {4},
  year    = {1992},
  pages   = {368--372},
  doi     = {10.1112/blms/24.4.368},
}

@Article{HH,
  author  = {Hadwin, D. and Hoover, T.},
  title   = {Operator algebras and the conjugacy of transformations},
  journal = {J. Funct. Anal},
  volume  = {77},
  number  = {1},
  year    = {1988},
  pages   = {112--122},
  doi     = {10.1016/0022-1236(88)90080-8},
}

@Book{KR19,
  author    = {Katsoulis, E. and Ramsey, C.},
  publisher = {AMS},
  title     = {Crossed products of operator algebras},
  series    = {Mem. Amer. Math. Soc.},
  volume    = {258},
  number = {1240},
  year      = {2019},
  doi ={https://doi.org/10.1090/memo/1240}
}

@Article{DK08,
    author={Davidson, K. and Katsoulis, E. G.},
    title={Isomorphisms between topological conjugacy algebras},
    volume={621},
    journal={J. Reine Angew. Math.},
    year={2008},
    pages={29-51}
}

@Article{DEG20,
    author={Dor-On, A. and Eilers, S. and Geﬀen, S.},
    title={Classification of irreversible and reversible Pimsner operator algebras},
    volume={156},
    journal={Compos. Math.},
    year={2020},
    pages={2510–2535}
}

@Book {Davlast,
    author = {Davidson, K.},
     title = {Functional analysis and operator algebras},
    series = {CMS/CAIMS Books in Mathematics},
    volume = {13},
 publisher = {Springer, Cham},
      year = {[2025] \copyright 2025},
     pages = {xiv+797},
      isbn = {978-3-031-63664-6; 978-3-031-63665-3},
   mrclass = {46-01 (46Lxx 47-01 47Lxx)},
  mrnumber = {4901168},
       doi = {10.1007/978-3-031-63665-3},
       url = {https://doi.org/10.1007/978-3-031-63665-3},
}

@Article {KakK1,
    author = {Kakariadis, E. T. A. and Katsoulis, E.},
     title = {Semicrossed products of operator algebras and their {${\rm
              C}^*$}-envelopes},
   journal = {J. Funct. Anal.},
  fjournaL = {Journal of Functional Analysis},
    volume = {262},
      year = {2012},
    number = {7},
     pages = {3108--3124},
      issn = {0022-1236,1096-0783},
   mrclass = {47L55 (46L07)},
  mrnumber = {2885949},
mrreviewer = {Benton\ L.\ Duncan},
       doi = {10.1016/j.jfa.2012.01.002},
       url = {https://doi.org/10.1016/j.jfa.2012.01.002},
}

@Article{BGR77,
    author={Brown, L. G. and Green, P. and Rieffel, M. A.},
    title={Stable isomorphism and strong Morita equivalence of {C}$^*$-algebras},
    volume={77},
    journal={Pac. J. Math.},
    year={1977},
    pages={349-363}
}

@Article{AJ69,
    author={Arveson, W. and Josephson, K.},
    title={Operator algebras and measure preserving automorphisms II},
    volume={4},
    journal={J. Funct. Anal.},
    year={1969},
    pages={100-134}
}

@Article{Hou10,
    author={Hou, C.},
    title={Cohomology of a class of Kadison-Singer algebras},
    volume={53},
    journal={Sci. China Math.},
    year={2010},
    pages={1827-1839}
}

@Article{GE16,
    author={Eleftherakis, G. K.},
    title={Stable isomorphism and strong Morita equivalence of operator algebras},
    volume={42},
    journal={Houston J. Math.},
    number={4}, 
    year={2016},
    pages={1245-1266}
}

@Article{EK17,
    author={Eleftherakis, G. K. and Kakariadis, E. T. A.},
    title={Strong Morita equivalence of operator spaces},
    volume={446},
    journal={J. Math. Anal. Appl.},
    number={2}, 
    year={2017},
    pages={1632-1653}
}

@Article{GE19,
    author={Eleftherakis, G. K.},
    title={On stable maps of operator algebras},
    volume={472},
    journal={J. Math. Anal. Appl.},
    number={2}, 
    year={2019},
    pages={1951-1975}
}

@Article{GE08,
    author={Eleftherakis, G. K.},
    title={A Morita type equivalence for dual operator algebras},
    volume={212},
    journal={J. Pure. Appl. Algebra},
    number={5}, 
    year={2008},
    pages={1060-1071}
}

@Article{GE10,
    author={Eleftherakis, G. K.},
    title={Morita type equivalences and reflexive algebras},
    volume={64},
    journal={J. Operator Theory},
    number={1}, 
    year={2010},
    pages={3-17}
}

@Article{EP08,
    author={Eleftherakis, G. K. and Paulsen, V. I.},
    title={Stably isomorphic dual operator algebras},
    volume={341},
    journal={Math. Ann.},
    number={1}, 
    year={2008},
    pages={99-112}
}

@Article{JP84,
    author={Peters, J.},
    title={Semicrossed products of {C}$^*$-algebras},
    volume={59},
    journal={J. Funct. Anal.},
    year={1984},
    pages={498-534}
}

@Article{DK11,
    author={Davidson, K. and Katsoulis, E. G.},
    title={Operator algebras for multivariable dynamics},
    journal={Mem. Amer. Math. Soc.},
    volume={209},
    year={2011},
    pages={viii+53 pp.}
}

@Article{KaKaLi23,
    author = {Kakariadis, E. T. A. and Katsoulis, E. and Li, X.},
    title = {Stable Isomorphisms of Operator Algebras},
    journal = {Int. Math. Res. Not.},
    volume = {2024},
    number = {5},
    pages = {4094-4118},
    year = {2023},
    issn = {1073-7928},
    doi = {10.1093/imrn/rnad146},
    url = {https://doi.org/10.1093/imrn/rnad146},
    eprint = {https://academic.oup.com/imrn/article-pdf/2024/5/4094/56901301/rnad146.pdf}
}

@Book{EL95,
  author    = {Lance, E. C.},
  publisher = {Cambridge University Press},
  title     = {Hilbert {C}$^*$-{Modules}, A Toolkit for Operator Algebraists},
  series    = {London Mathematical Society Lecture Note Series},
  volume    = {210},
  year      = {1995},
  doi = {https://doi.org/10.1017/CBO9780511526206},
}

@Book{WA76,
  author    = {Arveson, W.},
  publisher = {Springer Cham},
  title     = {An Invitation to {C}$^*$-Algebras},
  series    = {GTM},
  volume    = {39},
  year      = {1976},
}

@Book{KL17,
  author    = {Landsman, K.},
  publisher = {Springer Cham},
  title     = {Foundations of Quantum Theory: From Classical Concepts to Operator Algebras},
  series    = {Fundamental Theories of Physics},
  volume    = {188},
  year      = {2017},
  doi = {https://doi.org/10.1007/978-3-319-51777-3},
}

@Book{SS03,
  author    = {Stein, E. M. and Shakarchi, R.},
  publisher = {Priceton University Press},
  title     = {Fourier Analysis: An Introduction},
  series    = {Princeton Lectures in Analysis},
  volume    = {1},
  year      = {2003},
}

@Book{GM90,
  author    = {Murphy, G. J.},
  publisher = {Academic Press},
  title     = {C$^*$-Algebras and Operator Theory},
  year      = {1990},
  doi = {https://doi.org/10.1016/C2009-0-22289-6}
}

@Book{BlMe04,
  author    = {Blecher, D. P. and Le Merdy, C.},
  publisher = {Oxford University Press},
  title     = {Operator Algebras and Their Modules, An operator space approach},
  series    = {London Mathematical Society Monographs New Series},
  volume    = {30},
  year      = {2004},
}

@Book{KR97,
  author    = {Kadison, R. V. and Ringrose, J. R.},
  publisher = {American Mathematical Society},
  title     = {Fundamentals of the theory of operator algebras, Vol.\ I, II},
  series    ={Graduate Studies in Mathematics},
  volume    = {15, 21},
  year      = {1997},
}

\end{document}